\documentclass{amsart}
\usepackage{amssymb,latexsym,amsmath,amscd,graphicx,setspace,amsthm,verbatim,comment, stmaryrd}
\usepackage{tikz,tkz-berge}
\usepackage[margin = 3 cm]{geometry} \input xy \xyoption{all}

\usetikzlibrary{shapes.geometric}

\tikzset{every loop/.style={min distance=10mm,looseness=10}}

\newcount\mycount

\theoremstyle{plain}
\newtheorem{theorem}{Theorem}[section]
\newtheorem{proposition}[theorem]{Proposition}

\newtheorem{corollary}[theorem]{Corollary}
\newtheorem{lemma}[theorem]{Lemma}

\makeatletter
\def\th@remark{%
  \thm@headfont{\bfseries}%
  \normalfont 
  \thm@preskip\topsep \divide\thm@preskip\tw@
  \thm@postskip\thm@preskip
}
\makeatother

\theoremstyle{remark}

\theoremstyle{definition}

\def\Z{\mathbb{Z}}
\def\N{\mathbb{N}}

\begin{document} 
\title[On abelian  $\ell$-towers of multigraphs II]{On abelian  $\ell$-towers of multigraphs II}

\author{Kevin McGown, Daniel Valli\`{e}res}

\address{Mathematics and Statistics Department, California State University, Chico, CA 95929 USA} 
\email{kmcgown@csuchico.edu} 
\address{Mathematics and Statistics Department, California State University, Chico, CA 95929 USA}
\email{dvallieres@csuchico.edu}

\date{\today}

\begin{abstract}
Let $\ell$ be a rational prime.  Previously, abelian $\ell$-towers of multigraphs were introduced which are analogous to $\Z_{\ell}$-extensions of number fields.  It was shown that for a certain class of towers of bouquets, the growth of the $\ell$-part of the number of spanning trees behaves in a predictable manner (analogous to a well-known theorem of Iwasawa for $\Z_{\ell}$-extensions of number fields).  In this paper, we give a generalization to a broader class of regular abelian $\ell$-towers of bouquets than was originally considered.  To carry this out, we observe that certain shifted Chebyshev polynomials are members of a continuously parametrized family of power series with coefficients in $\Z_{\ell}$ and then study the special value at $s=1$ of the Artin-Ihara $L$-function $\ell$-adically.
\end{abstract} 
\maketitle 
\tableofcontents 

\section{Introduction}

In \cite{Vallieres:2021}, abelian $\ell$-towers of multigraphs were introduced
which can be viewed as being analogous to $\mathbb{Z}_{\ell}$-extensions of number fields, where $\ell$ is a rational prime.
To every tuple in $\Z_\ell^t$ (with $t\in\N$ and not all entries divisible by $\ell$) one can associate an abelian $\ell$-tower of a bouquet with $t$ loops.  Furthermore, it was proved that when the tuple belongs to $\Z^t$, the $\ell$-adic valuation of the number of spanning trees behaves similarly to the $\ell$-adic valuation of the class numbers in $\mathbb{Z}_{\ell}$-extensions of number fields, as in a well-known theorem of Iwasawa (see Theorem $11$ in \cite{Iwasawa:1959} and also \S$4.2$ of \cite{Iwasawa:1973}).
More specifically, if $\kappa_n$ denotes the number of spanning trees of the multigraph at the $n$-th level, then there exist non-negative integers $\mu, \lambda, n_{0}$ and an integer $\nu$ such that
$${\rm ord}_{\ell}(\kappa_n)=\mu\ell^n+\lambda n + \nu$$ for $n\geq n_0$. 
In the present paper, we extend this result to all of $\Z_\ell^t$,
thereby generalizing Theorem $5.6$ of \cite{Vallieres:2021} to a broader class of regular abelian $\ell$-towers of bouquets than was originally considered.

The paper is organized as follows.
In~\S\ref{shifted},
we observe that the coefficients of the shifted Chebyshev polynomials $P_a(T)$  employed in~\cite{Vallieres:2021} satisfy some congruences, and in~\S\ref{limits} we use this to show that there exists a continuous function
$$f:\mathbb{Z}_{\ell} \longrightarrow \mathbb{Z}_{\ell}\llbracket T \rrbracket, $$
satisfying $f(a)=P_a(T)$ for $a\in\N$.  This allows us, in~\S\ref{application}, to study $\ell$-adically the special value at $s=1$ of Artin-Ihara $L$-functions associated to the covers arising in abelian $\ell$-towers of a bouquet, which leads to our main result (see Theorem \ref{maintheorem}).  In \S \ref{examples}, we end the paper with a few examples.

\subsection*{Acknowledgement}
We would like to thank Thomas Mattman and John Lind for stimulating discussions and helpful suggestions.
\section{Shifted Chebyshev polynomials}\label{shifted}
Throughout this paper, $\mathbb{N}$ will denote the set of positive integers.  Recall that for $a \in \mathbb{Z}_{\ge 0}$ the Chebyshev polynomial (of the first type) $T_{a}(X)  \in \mathbb{Z}[X]$ is the unique polynomial satisfying
$$T_{a}(\cos(\theta)) = \cos(a \theta) $$
for all $\theta \in \mathbb{R}$.

In \S$5.4$ of \cite{Vallieres:2021}, some polynomials $P_{a}(X) \in \mathbb{Z}[X]$ were defined recursively as follows.  For $a = 0,1$, one sets $P_{0}(X) = 0$, $P_{1}(X) = X$, and for $a \ge 2$
$$P_{a}(X) = X(a^{2} - (a-1)P_{1}(X) -(a-2)P_{2}(X) - \ldots - P_{a-1}(X)).$$
Given $m \in \mathbb{N}$, we let $\zeta_{m} = \exp\left( 2 \pi i/m\right)$.  Furthermore, if $a \in \mathbb{Z}$, we set
$$\varepsilon_{m}(a) = (1 - \zeta_{m}^{a})(1 - \zeta_{m}^{-a}) \in \overline{\mathbb{Q}} \subseteq \mathbb{C}, $$
and we write $\varepsilon_{m}$ rather than $\varepsilon_{m}(1)$.
The polynomials $P_{a}(X)$ satisfy various properties, but most notably Lemma $5.5$ of \cite{Vallieres:2021} shows that for $m \in \mathbb{N}$ and $a \in \mathbb{Z}_{\ge 0}$, one has
\begin{equation} \label{identity}
P_{a}(\varepsilon_{m}) = \varepsilon_{m}(a).
\end{equation}
Consider now the polynomial $Q_{a}(X) =  2 - 2  \cdot T_{a} \left(1 - \frac{X}{2} \right)$.  Since
$$\varepsilon_{m}(a) = 2 - 2 \cos\left( \frac{2 \pi a}{m}\right), $$
we have $P_{a}(\varepsilon_{m}) = Q_{a}(\varepsilon_{m}) = \varepsilon_{m}(a)$ for all $m \in \mathbb{N}$.  Since $P_{a}(X)$ and $Q_{a}(X)$ agree at infinitely many points, they are equal.  Thus, the precise relationship between the polynomials $P_{a}(X)$ and the Chebyshev polynomials is
\begin{equation} \label{p_to_cheby}
P_{a}(X) =  2 - 2  \cdot T_{a} \left(1 - \frac{X}{2} \right),
\end{equation}
so that $P_{a}(X)$ is a shifted Chebyshev polynomial.  The polynomials $P_{a}(X)$ have no constant coefficients and have degree $a$.  From now on, we write
$$P_{a}(X) = d_{1}(a)X + d_{2}(a)X^{2} + \ldots + d_{a}(a)X^{a}. $$
In order to simplify the notation, we shall also make use of the falling and raising factorials
$$(X)_{n} = X(X-1) \ldots  (X - n +1) \text{ and } X^{(n)} = X(X+1) \ldots (X+n -1).$$
\begin{proposition} \label{exactfor}
For $k,n \in \mathbb{N}$, we have
\begin{equation*}
d_{k}(n) =
\begin{cases}
(-1)^{k-1} \binom{n+k-1}{2k-1} \frac{n}{k}, &\text{ if } k \le n;\\
0, &\text{ otherwise}.
\end{cases}
\end{equation*}
\end{proposition}
\begin{proof}
Our starting point (see \S $4.21$ of \cite{Szego:1975}) is the equality between Chebyshev polynomials and hypergeometric functions which gives
\begin{equation*}
\begin{aligned}
T_{n}(X) &= {}_{2}F_{1}\left( -n,n;\frac{1}{2},\frac{1-X}{2}\right)\\
&= \sum_{i=0}^{\infty} \frac{(-n)^{(i)}n^{(i)}}{\left( \frac{1}{2}\right)^{(i)}} \left( \frac{1}{i!}\right)\left( \frac{1-X}{2}\right)^{i}\\
&= 1 + n \sum_{i=1}^{n}(-2)^{i}\frac{(n + i -1)!}{(n-i)!(2i)!}(1-X)^{i}.
\end{aligned}
\end{equation*}
Combining with (\ref{p_to_cheby}) gives the desired result.
\end{proof}
In particular, we have
\begin{equation} \label{lead_coef}
d_{1}(n) = n^{2}  \text{ and } d_{n}(n) = (-1)^{n-1}. 
\end{equation}
The integers $d_{k}(n)$ satisfy certain congruence relations which we state in Proposition \ref{congruence} below, but we first need a lemma.
\begin{lemma} \label{ineq}
Let $\ell$ be a rational prime satisfying $\ell \ge 3$.  Then for all $t \in \mathbb{Z}_{\ge 0}$, we have
$${\rm ord}_{\ell}((2t + 1)! (t+1)) \le t $$
and
$${\rm ord}_{2}((2t + 1)! (t+1)) \le 2t. $$
\end{lemma}
\begin{proof}
Let $\ell$ be an arbitrary prime number.  The claim is true if $t = 0$ so we can assume that $t \ge 1$.  To simplify the notation, let $v = {\rm ord}_{\ell}(t+1)$ and $N = \lfloor \log_{\ell}(2t + 1) \rfloor$.   Note first that
$$v \le \lfloor \log_{\ell}(t+1)\rfloor \le N. $$
Using Legendre's formula, we have
$${\rm ord}_{\ell}((2t + 1)!(t+1)) = \sum_{k=1}^{v}\left\lfloor \frac{2t+1}{\ell^{k}} \right\rfloor + \sum_{k = v+1}^{N} \left\lfloor \frac{2t+1}{\ell^{k}}\right\rfloor + v, $$
where the first sum is interpreted as being zero when $v=0$ and the second sum is interpreted as being zero when $v = N$.  Now, we have
\begin{equation*}
\begin{aligned}
{\rm ord}_{\ell}((2t + 1)!(t+1)) &\le  \sum_{k=1}^{v}\left\lfloor \frac{2t+1}{\ell^{k}} \right\rfloor + \sum_{k = v+1}^{N}  \frac{2t+1}{\ell^{k}} + v \\
&=  \sum_{k=1}^{v}\left\lfloor \frac{2(t+1)}{\ell^{k}} - \frac{1}{\ell^{k}} \right\rfloor + \sum_{k = v+1}^{N}  \frac{2t+1}{\ell^{k}} + v \\
&=  \sum_{k=1}^{v}\frac{2(t+1)}{\ell^{k}} + \sum_{k=1}^{v}\left\lfloor - \frac{1}{\ell^{k}} \right\rfloor + \frac{2t + 1}{(\ell - 1)\cdot  \ell^{v}} \left(1-\frac{1}{\ell^{N - v}} \right) + v \\
&= \sum_{k=1}^{v}\frac{2(t+1)}{\ell^{k}}  + \frac{2t + 1}{(\ell - 1)\cdot  \ell^{v}} \left(1-\frac{1}{\ell^{N - v}} \right) \\
&= \frac{2(t+1)}{\ell - 1}\left(1 - \frac{1}{\ell^{v}}\right) + \frac{2t + 1}{(\ell - 1)\cdot  \ell^{v}} \left(1-\frac{1}{\ell^{N - v}} \right).
\end{aligned}
\end{equation*}
If $\ell \ge 3$, then we continue as follows
\begin{equation*}
\begin{aligned}
&\le (t+1)\left(1 - \frac{1}{\ell^{v}}\right) + \left( t + \frac{1}{2}\right)\frac{1}{\ell^{v}}  \\
&= (t+1) - \frac{1}{2 \cdot \ell^{v}} \\
&< t+1,
\end{aligned}
\end{equation*}
whereas if $\ell = 2$, we rather continue as follows
\begin{equation*}
\begin{aligned}
&= 2(t+1)\left(1 - \frac{1}{2^{v}} \right) + (2t +1)\left(\frac{1}{2^{v}} - \frac{1}{2^{N}} \right) \\
&= 2t + 2 - \frac{1}{2^{v}} - \frac{2t + 1}{2^{N}} \\
&< 2t + 1,
\end{aligned}
\end{equation*}
since
$$\frac{2t+1}{2^{N}} \ge 1. $$
This ends the proof.
\end{proof}

\begin{proposition} \label{congruence}
Let $\ell$ be any rational prime and let $m, n \in \mathbb{N}$ be such that $$m \equiv n \pmod{\ell^{s}}$$ for some $s \in \mathbb{N}$.  Then, we have
$$d_{t+1}(m) \equiv d_{t+1}(n) \pmod{\ell^{s-t}}, $$
for $t = 0,1,\ldots, s-1$.  
\end{proposition}
\begin{proof}
Without lost of generality, let us assume that $n \ge m$.  Assuming first $\ell \ge 3$, it follows from Proposition \ref{exactfor} that for $a \ge t+1$, we have
\begin{equation} \label{usefulid}
d_{t+1}(a) = (-1)^{t} \frac{(a + t)_{2t+1}}{(2t+1)!} \frac{a}{t+1}.
\end{equation}
Write $n = m + q\ell^{s}$ for some integer $q$.  If $m \le n < t+1$, then the claim is obviously true.  If now $m< t+1 \le n$, then we have to show 
\begin{equation} \label{cong1}
d_{t+1}(n) \equiv 0 \pmod{\ell^{s-t}}. 
\end{equation}
By (\ref{usefulid}), we have
$$d_{t+1}(n) = (-1)^{t} \frac{(m + q\ell^{s} + t)_{2t+1}}{(2t+1)!} \frac{n}{t+1}.$$
Since $t+1 >m$, we have $0 \le m+t \le 2t$ and thus $ (m + q\ell^{s} + t)_{2t+1} \equiv 0 \pmod{\ell^{s}}$.  Combining this with Lemma \ref{ineq} gives
$${\rm ord}_{\ell}(d_{t+1}(n)) \ge s - t,$$
which implies (\ref{cong1}).  

If we have $1 \le t+1 \le m \le n$, then
\begin{equation} \label{useful2}
d_{t+1}(n) - d_{t+1}(m) =  \frac{(-1)^{t}}{t+1}\left(m\left(\binom{n+t}{2t+1} - \binom{m +t}{2t+1} \right) + \binom{n+t}{2t+1}q\ell^{s} \right),
\end{equation}
but
$$\binom{n+t}{2t+1} - \binom{m+t}{2t+1} = \frac{(m+q\ell^{s} + t)_{2t+1} - (m + t)_{2t+1}}{(2t+1)!}. $$
Now, $(m+q\ell^{s} + t)_{2t+1} \equiv (m + t)_{2t+1} \pmod{\ell^{s}}$.  Therefore combining this with (\ref{useful2}) and Lemma \ref{ineq} gives the desired congruence as well.

If $\ell = 2$, we first consider $m<t+1\leq n$.  In this case, one of the terms in the product
$$(m + q2^{s} + t)_{2t+1}= (m + q2^{s} + t) \ldots (m + q 2^{s} - t), $$
is equal to $q2^s$,
and moreover, there are at least $t$ even terms
and at least $2$ terms divisible by $4$ when $t\geq 2$.
This gives 
$${\rm ord}_{2}((m + q2^{s} + t)_{2t+1}) \ge s + t $$
unless $t=2$ in which case ${\rm ord}_{2}((m + q2^{s} + t)_{2t+1})\ge s + t -1$.  Combining this with Lemma \ref{ineq} and the equality ${\rm ord}_{2}(5! \cdot 3) = 3$ gives (\ref{cong1}).

At last, we consider $1\leq t+1\leq m\leq n$
where
$$m\left(\binom{n+t}{2t+1} - \binom{m+t}{2t+1}\right) = \frac{m(m+q2^{s} + t)_{2t+1} - m(m + t)_{2t+1}}{(2t+1)!}. $$
It suffices to prove that
$${\rm ord}_2(m(m+q2^{s} + t)_{2t+1} - m(m + t)_{2t+1})\geq s+t.$$
The product $m(m+t)_{2t+1}$ contains exactly $t+1$ even integers, as does the product $m(m+t+q 2^s)_{2t+1}$.
After factoring $2^{t+1}$ out of each product, the remaining difference is visibly congruent to zero mod $2^{s-1}$; indeed,
a typical term from the first product looks like $(m+j)/2$ and the corresponding term in the second product looks like $(m+j)/2+q2^{s-1}$.  Again combining this with (\ref{useful2}) and Lemma \ref{ineq} gives the desired congruence.

\end{proof}

\section{$\ell$-adic limits of shifted Chebyshev polynomials}\label{limits}
Our main references for this section are \cite{Iwasawa:1973}, \cite{Neukirch:1999} and \cite{Robert:2000}.  We start by fixing a rational prime $\ell$.  The unital commutative ring $\mathbb{Z}_{\ell}\llbracket T \rrbracket$ of power series with coefficient in $\mathbb{Z}_{\ell}$ is a local ring, and its unique maximal ideal is given by $\mathfrak{m} = (\ell,T)$.  We view $\mathbb{Z}_{\ell}\llbracket T \rrbracket$ with its usual topology given by the filtration
$$\mathfrak{m} \supseteq \mathfrak{m}^{2} \supseteq \mathfrak{m}^{3} \supseteq \ldots \supseteq \mathfrak{m}^{n} \supseteq \ldots $$
For $Q \in \mathbb{Z}_{\ell}\llbracket T \rrbracket$, let
\begin{equation*}
||Q||=
\begin{cases}
\ell^{-v}, &\text{ if } Q \neq 0;\\
0, &\text{ if } Q = 0,
\end{cases}
\end{equation*}
where $v = {\rm max}\{ n \in \mathbb{Z}_{\ge 0} \, | \, Q \in \mathfrak{m}^{n}\}$ with the understanding that $\mathfrak{m}^{0}$ is the full ring.  Then, for $Q,R \in \mathbb{Z}_{\ell}\llbracket T \rrbracket$, the function $||\cdot||:\mathbb{Z}_{\ell}\llbracket T\rrbracket \longrightarrow \mathbb{R}_{\ge 0}$ satisfies:
\begin{enumerate}
\item $||Q||=0$ if and only if $Q = 0$,
\item $||Q \cdot R|| \le ||Q||\cdot ||R||$,
\item $||Q + R|| \le ||Q|| + ||R||$.
\end{enumerate}
Thus, the pair $(\mathbb{Z}_{\ell}\llbracket T\rrbracket,d)$, where $d(Q,R) = || Q - R||$, is a metric space.  As such, $\mathbb{Z}_{\ell}\llbracket T \rrbracket$ is a complete local ring.  

We have a function $f:\mathbb{N} \longrightarrow \mathbb{Z}_{\ell}\llbracket T \rrbracket$ defined by
$$n \mapsto f(n) = P_{n}(T). $$
\begin{corollary}
The function $f$ is uniformly continuous when $\mathbb{N}$ is endowed with the $\ell$-adic topology.
\end{corollary}
\begin{proof}
It follows from Proposition \ref{congruence} that given any $N \in \mathbb{N}$ and any $m,n \in \mathbb{N}$ satisfying 
$${\rm ord}_{\ell}(m-n) \ge N,$$ 
one has
$$P_{m}(T) - P_{n}(T) \in \mathfrak{m}^{N+1}, $$
and this shows the claim.
\end{proof}
Since $\mathbb{N}$ is dense in $\mathbb{Z}_{\ell}$, the function $f$ can be uniquely extended to a continuous function 
$$\mathbb{Z}_{\ell} \longrightarrow \mathbb{Z}_{\ell}\llbracket T \rrbracket $$
which we denote by the same symbol $f$.  If $a \in \mathbb{Z}_{\ell}$, then we let 
$$P_{a} (T) = f(a) \in \mathbb{Z}_{\ell}\llbracket T \rrbracket. $$
For example, if we take $\ell = 5$ and $a = 1/3 \in \mathbb{Z}_{5}$, then we have
$$P_{1/3}(T) = (4.2012\ldots) T + (4.2342\ldots) T^{2} + (2.2130\ldots) T^{3} + (0.3400\ldots) T^{4} + \ldots \in \mathbb{Z}_{5}\llbracket T \rrbracket. $$
We note that the power series $P_{a}(T)$ begins as follows
$$P_{a}(T) = a^{2}T + \ldots,$$
by (\ref{lead_coef}).

Consider now the field
$$\mathbb{Q}(\zeta_{\ell^{\infty}}) = \mathbb{Q}(\zeta_{\ell^{i}} \, | \, i=1,2,\ldots) \subseteq \mathbb{C}. $$
It is an infinite algebraic extension of $\mathbb{Q}$.  From now on, we fix an embedding
$$\tau: \mathbb{Q}(\zeta_{\ell^{\infty}}) \hookrightarrow \overline{\mathbb{Q}}_{\ell}. $$
For $i=0,1,2,\ldots $ , we let
$$\xi_{\ell^{i}} = \tau(\zeta_{\ell^{i}}) \in \overline{\mathbb{Q}}_{\ell}. $$
Furthermore, if $a \in \mathbb{Z}$, we let
$$\eta_{\ell^{i}}(a) = \tau(\varepsilon_{\ell^{i}}(a)) = (1 - \xi_{\ell^{i}}^{a})(1- \xi_{\ell^{i}}^{-a}) \in \overline{\mathbb{Q}}_{\ell}, $$
and we write $\eta_{\ell^{i}}$ instead of $\eta_{\ell^{i}}(1)$.  Recall that the valuation ${\rm ord}_{\ell}$ as well as the absolute value $|\cdot |_{\ell}$ on $\mathbb{Q}$ can be extended uniquely to $\overline{\mathbb{Q}}_{\ell}$ and also all the way up to $\mathbb{C}_{\ell}$.  We will denote the valuation on $\mathbb{C}_{\ell}$ by $v_{\ell}$ and the absolute value by $|\cdot|_{\ell}$.  They are related to one another via $v_{\ell}(x) = - \log_{\ell}(|x|_{\ell})$ for all $x \in \mathbb{C}_{\ell}$.

If $x \in \mathbb{Q}(\zeta_{\ell^{i}})$ and if $\mathcal{L}_{i}$ is the unique prime ideal of $\mathbb{Q}(\zeta_{\ell^{i}})$ lying above $\ell$, we have
\begin{equation} \label{rel_bet_val}
v_{\ell}(\tau(x)) = \frac{1}{\varphi(\ell^{i})} {\rm ord}_{\mathcal{L}_{i}}(x),
\end{equation}
for all $x \in \mathbb{Q}(\zeta_{\ell^{i}})$.  Here $\varphi$ denotes the Euler $\varphi$-function and ${\rm ord}_{\mathcal{L}_{i}}$ is the valuation on $\mathbb{Q}(\zeta_{\ell^{i}})$ associated to the prime ideal $\mathcal{L}_{i}$.  Since ${\rm ord}_{\mathcal{L}_{i}}(1-\zeta_{\ell^{i}}) = 1$, we have
\begin{equation} \label{val_of_ele}
v_{\ell}(\eta_{\ell^{i}}) = \frac{2}{\varphi(\ell^{i})} \text{ and } |\eta_{\ell^{i}}|_{\ell} = \ell^{-2/\varphi(\ell^{i})}.
\end{equation}

From now on, we let
$$D = \{x \in \mathbb{C}_{\ell} \, : \, |x|_{\ell} < 1 \}. $$
Fix $i \in \mathbb{N}$, and let $\alpha = \xi_{\ell^{i}} - 1$.  We have $|\alpha|_{\ell} = \ell^{-1/\varphi(\ell^{i})} < 1$, so that $\alpha \in D$ and thus $|\alpha|_{\ell}^{n} \to 0$ as $n \to \infty$.  By the theory of Mahler series (see Chapter $4$ of \cite{Robert:2000}), the function $\mathbb{N} \longrightarrow \mathbb{C}_{\ell}$ defined via 
$$a \longrightarrow \xi_{\ell^{i}}^{a} $$
extends to a continuous function $\mathbb{Z}_{\ell} \longrightarrow \mathbb{C}_{\ell}$ which we denote by the same symbol.  Note that if
$$a = \sum_{k=0}^{\infty}a_{k}\ell^{k} \in \mathbb{Z}_{\ell}, \, \, a_{k} \in \{0,1,\ldots,\ell-1 \},$$
then
\begin{equation} \label{concrete}
\xi_{\ell^{i}}^{a} = \xi_{\ell^{i}}^{\sum_{k=0}^{i-1}a_{k}\ell^{k}}. 
\end{equation}
We obtain for each $i \in \mathbb{N}$ a continuous function $\mathbb{Z}_{\ell} \longrightarrow \mathbb{C}_{\ell}$ defined via
$$a \mapsto \eta_{\ell^{i}}(a) =  (1 - \xi_{\ell^{i}}^{a})(1- \xi_{\ell^{i}}^{-a}). $$
From (\ref{concrete}), we actually have $\eta_{\ell^{i}}(\mathbb{Z}_{\ell}) \subseteq \overline{\mathbb{Q}}_{\ell}$.

Given any $Q(T) = \sum_{k=0}^{\infty}a_{k}T^{k} \in \mathbb{Z}_{\ell}\llbracket T \rrbracket$, it defines a continuous function $Q:D \longrightarrow \mathbb{C}_{\ell}$ via
$$x \mapsto Q(x) = \sum_{k=0}^{\infty}a_{k}x^{k}. $$
By (\ref{val_of_ele}), we thus get for every $a \in \mathbb{Z}_{\ell}$ a well-defined number $P_{a}(\eta_{\ell^{i}}) \in \mathbb{C}_{\ell}$.
\begin{lemma} \label{cont}
For each $i \in \mathbb{N}$, the function ${\rm ev}_{\eta_{\ell^{i}}}:\mathbb{Z}_{\ell}\llbracket T \rrbracket \longrightarrow \mathbb{C}_{\ell}$ given by
$$Q(T) \mapsto {\rm ev}_{\eta_{\ell^{i}}}(Q(T)) = Q(\eta_{\ell^{i}}) $$
is uniformly continuous.
\end{lemma}
\begin{proof}
It suffices to show that for all $\varepsilon > 0$ there exists $N \in \mathbb{N}$ such that if $Q(T) \in \mathfrak{m}^{N}$, then $|Q(\eta_{\ell^{i}})|_{\ell} < \varepsilon$.  To simplify the notation, let $x = |\eta_{\ell^{i}}|_{\ell} = \ell^{-2/\varphi(\ell^{i})} < 1$.  
If 
$$Q(T) = \sum_{k=0}^{\infty}a_{k}T^{k} \in \mathfrak{m}^{N}, $$
then $|a_{k}|_{\ell} \le \ell^{k-N}$ for $k=0,\ldots,N$, and thus
\begin{equation*}
\begin{aligned}
\Big|\sum_{k=0}^{\infty}a_{k}(\eta_{\ell^{i}})^{k}\Big|_{\ell} &\le \sum_{k=0}^{N-1}|a_{k}|_{\ell} \cdot x^{k} + \sum_{k=N}^{\infty}x^{k} \\
&\le \frac{1}{\ell^{N}}\sum_{k=0}^{N-1}(\ell x)^{k} + \frac{ x^{N}}{1-x},
\end{aligned}
\end{equation*}
which can be made arbitrarily small for $N$ large. 
\end{proof}
As a consequence, we obtain the following result.
\begin{corollary} \label{key_prop}
Given $a \in \mathbb{Z}_{\ell}$ and any $i \in \mathbb{Z}_{\ge 0}$, we have $P_{a}(\eta_{\ell^{i}}) = \eta_{\ell^{i}}(a)$.
\end{corollary}
\begin{proof}
If $i=0$, the equality is clear.  If $i \ge 1$, then the function ${\rm ev}_{\eta_{\ell^{i}}} \circ f:\mathbb{Z}_{\ell} \longrightarrow \mathbb{C}_{\ell}$ is continuous by Lemma \ref{cont}.  Furthermore, the function $\mathbb{Z}_{\ell} \longrightarrow \mathbb{C}_{\ell}$ given by $a \mapsto \eta_{\ell^{i}}(a)$ is also continuous as we pointed out before.  Since both these functions agree on $\mathbb{N}$ by (\ref{identity}) and $\mathbb{N}$ is dense in $\mathbb{Z}_{\ell}$, the claim follows.
\end{proof}

\section{Abelian $\ell$-towers of bouquets}\label{application}
For this section, we assume that the reader is familiar with \cite{Vallieres:2021} and in particular with the notion of an abelian $\ell$-tower of multigraphs.  (See Definition $4.1$ of \cite{Vallieres:2021}.)  Recall that if $S$ is a finite set and 
$$i:S \longrightarrow \mathbb{Z}_{\ell} $$
is any function for which there exists $s \in S$ such that $i(s) \in \mathbb{Z}_{\ell}^{\times}$, then one gets a regular abelian $\ell$-tower of connected multigraphs
$$X = B_{|S|} \longleftarrow X(\mathbb{Z}/\ell \mathbb{Z},S,i_{1}) \longleftarrow X(\mathbb{Z}/\ell^{2}\mathbb{Z},S,i_{2}) \longleftarrow \ldots \longleftarrow X(\mathbb{Z}/\ell^{n}\mathbb{Z},S,i_{n}) \longleftarrow \ldots,$$
where $B_{t}$ denotes a bouquet with $t$ loops and $X(\mathbb{Z}/\ell^{n}\mathbb{Z},S,i_{n})$ is the Cayley-Serre multigraph associated to the data $(\mathbb{Z}/\ell^{n}\mathbb{Z},S,i_{n})$.  The function $i_{n}$ is the one obtained from the composition
$$S \stackrel{i}{\longrightarrow} \mathbb{Z}_{\ell} \longrightarrow \mathbb{Z}_{\ell}/\ell^{n}\mathbb{Z}_{\ell} \stackrel{\simeq}{\longrightarrow} \mathbb{Z}/\ell^{n}\mathbb{Z}. $$
Theorem $5.6$ of \cite{Vallieres:2021} applies to regular abelian $\ell$-towers as above in the case where
$$i(S) \subseteq \mathbb{Z}. $$
We can now remove this condition.  (The case $|S| = 1$ has already been treated separately.  See the discussion after Definition $4.1$ of \cite{Vallieres:2021}.)
\begin{theorem}\label{maintheorem}
Let $S =\{s_{1},\ldots,s_{t} \}$ be a finite set with cardinality $t \ge 2$, and let $i:S \longrightarrow \mathbb{Z}_{\ell}$ be a function.  For $j=1,\ldots, t$, let $a_{j}$ be the $\ell$-adic integer satisfying $i(s_{j})=a_{j}$.  Assume that at least one of $a_{1},\ldots,a_{t}$ is relatively prime with $\ell$ (so that our Cayley-Serre multigraphs are connected).  Consider the regular abelian $\ell$-tower
$$X = B_{t} \longleftarrow X(\mathbb{Z}/\ell \mathbb{Z},S,i_{1}) \longleftarrow X(\mathbb{Z}/\ell^{2}\mathbb{Z},S,i_{2}) \longleftarrow \ldots \longleftarrow X(\mathbb{Z}/\ell^{n}\mathbb{Z},S,i_{n}) \longleftarrow \ldots$$
and define the $\ell$-adic integers $c_{j}$ via 
\begin{equation*}
\begin{aligned}
Q(T) &= P_{a_{1}}(T) + \ldots + P_{a_{t}}(T) \\
&= c_{1}T + c_{2}T^{2} + \ldots \in \mathbb{Z}_{\ell}\llbracket T \rrbracket.
\end{aligned}
\end{equation*}
Let
$$\mu = {\rm min}\{v_{\ell}(c_{j})\, | \, j =1,2,\ldots\}, $$
and 
$$\lambda = {\rm min}\{2j \, | \, j \in \mathbb{N} \text{ and } v_{\ell}(c_{j}) = \mu\}  -1. $$
If $\kappa_{n}$ denotes the number of spanning trees of $X(\mathbb{Z}/\ell^{n}\mathbb{Z},S,i_{n})$, then there exist a nonnegative integer $n_{0}$ and a constant $\nu \in \mathbb{Z}$ (depending also on the $a_{j}$) such that
$${\rm ord}_{\ell}(\kappa_{n}) = \mu \ell^{n} + \lambda n + \nu,$$
when $n \ge n_{0}$.  
\end{theorem}
\begin{proof}
Using (\ref{rel_bet_val}), the first equation in the proof of Theorem $5.6$ in \cite{Vallieres:2021} becomes
$${\rm ord}_{\ell}(\kappa_{n}) = -n + \sum_{i = 1}^{n}\varphi(\ell^{i})v_{\ell}(\eta_{\ell^{i}}(a_{1}) + \ldots + \eta_{\ell^{i}}(a_{t})).$$
By Corollary \ref{key_prop}, we have
\begin{equation*}
\begin{aligned}
\eta_{\ell^{i}}(a_{1}) + \ldots + \eta_{\ell^{i}}(a_{t}) &= P_{a_{1}}(\eta_{\ell^{i}}) + \ldots + P_{a_{t}}(\eta_{\ell^{i}})\\
&= Q(\eta_{\ell^{i}}).
\end{aligned}
\end{equation*}
We claim that for $i$ large
\begin{equation} \label{claim}
v_{\ell}\left(Q(\eta_{\ell^{i}}) \right) = \mu + \frac{\lambda + 1}{\varphi(\ell^{i})}.
\end{equation}
Let us write $Q(T) = \ell^{\mu} \cdot R(T)$ for some
$$R(T) = \sum_{k=1}^{\infty} e_{k}T^{k} \in \mathbb{Z}_{\ell}\llbracket T \rrbracket \smallsetminus \ell \mathbb{Z}_{\ell}\llbracket T \rrbracket. $$
To prove (\ref{claim}), it suffices to show
\begin{equation}\label{claim2}
v_{\ell}\left( R(\eta_{\ell^{i}})\right) = \frac{\lambda+1}{\varphi(\ell^{i})}, 
\end{equation}
when $i$ is large.  Let us define
$$k_{0} = {\rm min}\{k\in\mathbb{N} \mid v_\ell(e_k)=0\} $$
so that we have $\lambda+1 = 2k_{0}$.
We assume $i$ is large enough so that
$$
  \frac{2k_0}{\varphi(\ell^i)}\leq 1
$$
as this ensures that the values
$$
  v_\ell(e_k(\eta_{\ell^i})^k)=v_\ell(e_k)+\frac{2k}{\varphi(\ell^i)}
$$
are distinct for $k=1,\dots, k_0$.  (Indeed, if two such values were equal with $1\leq j<k\leq k_0$, then $2(k-j)/\varphi(\ell^i)\in\mathbb{Z}$ and hence
$\varphi(\ell^i)\leq 2(k-j)<2k_0$.)  From this, together with the fact that for $k<k_0$ one has
$$
  v_\ell(e_k(\eta_{\ell^i})^k)>v_\ell(e_k)\geq 1
  \,,
$$
we obtain
$$
  v_\ell\left(\sum_{k=1}^{k_0}e_k(\eta_{\ell^i})^k\right)=\frac{2k_0}{\varphi(\ell^i)}=\frac{\lambda+1}{\varphi(\ell^i)}
  \,.
$$
In addition, we have
$$
  v_\ell\left(\sum_{k=k_0+1}^\infty e_k(\eta_{l^i})^k\right)
  \geq \min_{k> k_0}\frac{2k}{\varphi(\ell^i)}=\frac{2(k_0+1)}{\varphi(\ell^i)}
$$
and (\ref{claim2}) follows.

Therefore, there exists $n_{0} \ge 0$ and an integer $C$ such that if $n \ge n_{0}$, then
\begin{equation*}
\begin{aligned}
{\rm ord}_{\ell}(\kappa_{n}) &= -n + C +  \sum_{i = n_{0}}^{n}\varphi(\ell^{i})v_{\ell}(Q(\eta_{\ell^{i}})) \\
&= -n + C + \sum_{i = n_{0}}^{n}(\mu \cdot \varphi(\ell^{i}) + (\lambda+1))\\
&= -n + C + (\lambda+1)(n-(n_{0}-1)) + \mu(\ell^{n} - \ell^{n_{0} - 1}),
\end{aligned}
\end{equation*}
and this ends the proof.
\end{proof}
From the proof of the last theorem, it follows that any $n_{0} \in \mathbb{N}$ satisfying
$$n_{0} \ge \log_{\ell} \left(\frac{\ell}{\ell - 1} (\lambda + 1) \right) $$
will work.

In addition, we point out that since $c_{1} = a_{1}^{2} + \ldots + a_{t}^{2}$ by (\ref{lead_coef}), if 
$$\ell \nmid a_{1}^{2}+ \ldots + a_{t}^{2},$$
then we have $\mu=0, \lambda=1$ and $\nu = 0$.

\section{Examples}\label{examples}
The computations of the number of spanning trees in this section have been performed with the software \cite{SAGE}.  The computations of the power series $P_{a}(T)$ have been performed with the software \cite{PARI}.

\begin{enumerate}
\item Let $a_{1} =1/3, a_{2} = 3/5$ and $\ell = 2$.  Then, we get:
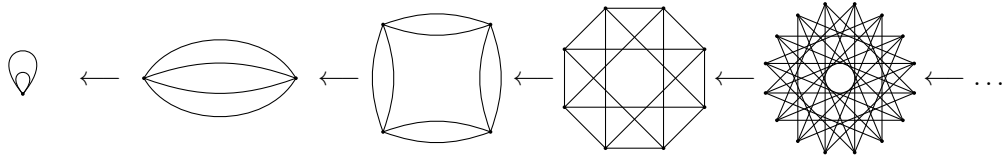
\begin{equation*}
\begin{tikzpicture}[baseline={([yshift=-1.7ex] current bounding box.center)}]
\node[draw=none,minimum size=2cm,regular polygon,regular polygon sides=1] (a) {};
\foreach \x in {1}
  \fill (a.corner \x) circle[radius=0.7pt];
\draw (a.corner 1) to [in=50,out=130,loop] (a.corner 1);
\draw (a.corner 1) to [in=50,out=130,distance = 0.5cm,loop] (a.corner 1);
\end{tikzpicture}
\longleftarrow \, \, \,
\begin{tikzpicture}[baseline={([yshift=-0.6ex] current bounding box.center)}]
\node[draw=none,minimum size=2cm,regular polygon,rotate = -45,regular polygon sides=4] (a) {};

  \fill (a.corner 1) circle[radius=0.7pt];
  \fill (a.corner 3) circle[radius=0.7pt];
  
  \path (a.corner 1) edge [bend left=20] (a.corner 3);
  \path (a.corner 1) edge [bend left=60] (a.corner 3);
  \path (a.corner 1) edge [bend right=20] (a.corner 3);
  \path (a.corner 1) edge [bend right=60] (a.corner 3);

\end{tikzpicture}
\, \, \, \longleftarrow 
\begin{tikzpicture}[baseline={([yshift=-0.6ex] current bounding box.center)}]
\node[draw=none,minimum size=2cm,regular polygon,regular polygon sides=4] (a) {};

\foreach \x in {1,2,...,4}
  \fill (a.corner \x) circle[radius=0.7pt];

\path (a.corner 1) edge [bend left=20] (a.corner 2);
\path (a.corner 1) edge [bend right=20] (a.corner 2);
\path (a.corner 2) edge [bend left=20] (a.corner 3);
\path (a.corner 2) edge [bend right=20] (a.corner 3);
\path (a.corner 3) edge [bend left=20] (a.corner 4);
\path (a.corner 3) edge [bend right=20] (a.corner 4);
\path (a.corner 4) edge [bend left=20] (a.corner 1);
\path (a.corner 4) edge [bend right=20] (a.corner 1);
\end{tikzpicture}
\longleftarrow
\begin{tikzpicture}[baseline={([yshift=-0.6ex] current bounding box.center)}]
\node[draw=none,minimum size=2cm,regular polygon,regular polygon sides=8] (a) {};

\foreach \x in {1,2,...,8}
  \fill (a.corner \x) circle[radius=0.7pt];
  
\foreach \y\z in {1/4,2/5,3/6,4/7,5/8,6/1,7/2,8/3}
  \path (a.corner \y) edge (a.corner \z);
  
\foreach \y\z in {1/8,2/1,3/2,4/3,5/4,6/5,7/6,8/7}
  \path (a.corner \y) edge (a.corner \z);
\end{tikzpicture}
\longleftarrow 
\begin{tikzpicture}[baseline={([yshift=-0.6ex] current bounding box.center)}]
\node[draw=none,minimum size=2cm,regular polygon,regular polygon sides=16] (a) {};

\foreach \x in {1,2,...,16}
  \fill (a.corner \x) circle[radius=0.7pt];
  
\foreach \y\z in {1/12,2/13,3/14,4/15,5/16,6/1,7/2,8/3,9/4,10/5,11/6,12/7,13/8,14/9,15/10,16/11}
  \path (a.corner \y) edge (a.corner \z);
  
\foreach \y\z in {1/8,2/9,3/10,4/11,5/12,6/13,7/14,8/15,9/16,10/1,11/2,12/3,13/4,14/5,15/6, 16/7}
  \path (a.corner \y) edge  (a.corner \z);
\end{tikzpicture}
\longleftarrow \ldots
\end{equation*}
The power series $Q$ starts as follows 
$$Q(T) = (0.1010\ldots) T + (0.1000\ldots) T^{2} + (1.0101\ldots) T^{3} + (0.0000\ldots) T^{4} + \ldots \in \mathbb{Z}_{2}\llbracket T \rrbracket,$$
so we should have $\mu = 0$ and $\lambda = 5$.  We calculate
$$\kappa_{0} = 1, \kappa_{1} = 2^{2}, \kappa_{2} = 2^{5}, \kappa_{3} = 2^{12}, \kappa_{4} = 2^{17}\cdot 17^{2}, \kappa_{5} = 2^{22} \cdot 17^{2} \cdot 1217^{2}, \ldots $$
We have
$${\rm ord}_{2}(\kappa_{n}) =  5n -3,$$
for all $n \ge 3$.

\item Let $a_{1}=1/2, a_{2}=1/5, a_{3} =1/7$ and $\ell = 3$.  Then, we get:
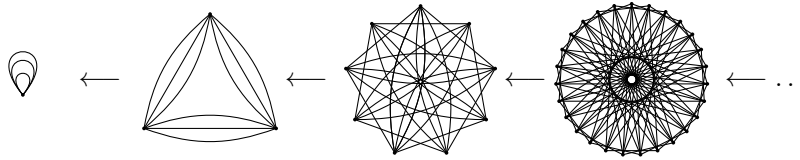
\begin{equation*}
\begin{tikzpicture}[baseline={([yshift=-1.7ex] current bounding box.center)}]
\node[draw=none,minimum size=3cm,regular polygon,regular polygon sides=1] (a) {};
\foreach \x in {1}
  \fill (a.corner \x) circle[radius=0.7pt];
\draw (a.corner 1) to [in=50,out=130,loop] (a.corner 1);
\draw (a.corner 1) to [in=50,out=130,distance = 0.8cm,loop] (a.corner 1);
\draw (a.corner 1) to [in=50,out=130,distance = 0.5cm,loop] (a.corner 1);
\end{tikzpicture}
\longleftarrow \, \, \,
\begin{tikzpicture}[baseline={([yshift=-0.6ex] current bounding box.center)}]
\node[draw=none,minimum size=2cm,regular polygon,regular polygon sides=3] (a) {};

\foreach \x in {1,2,3}
  \fill (a.corner \x) circle[radius=0.7pt];

\path (a.corner 1) edge [bend left=20] (a.corner 2);
\path (a.corner 1) edge [bend right=20] (a.corner 2);
\path (a.corner 2) edge [bend left=20] (a.corner 3);
\path (a.corner 2) edge [bend right=20] (a.corner 3);
\path (a.corner 3) edge [bend left=20] (a.corner 1);
\path (a.corner 3) edge [bend right=20] (a.corner 1);

\path (a.corner 1) edge  (a.corner 2);
\path (a.corner 2) edge  (a.corner 3);
\path (a.corner 3) edge  (a.corner 1);

\end{tikzpicture}
\longleftarrow \, \,
\begin{tikzpicture}[baseline={([yshift=-0.6ex] current bounding box.center)}]
\node[draw=none,minimum size=2cm,regular polygon,regular polygon sides=9] (a) {};

\foreach \x in {1,2,...,9}
  \fill (a.corner \x) circle[radius=0.7pt];
  
\foreach \y\z in {1/3,2/4,3/5,4/6,5/7,6/8,7/9,8/1,9/2}
  \path (a.corner \y) edge (a.corner \z);

  \path (a.corner 1) edge [bend left=20] (a.corner 5);
  \path (a.corner 1) edge [bend right=20] (a.corner 5);

  \path (a.corner 1) edge [bend left=20] (a.corner 6);
  \path (a.corner 1) edge [bend right=20] (a.corner 6);

  \path (a.corner 2) edge [bend left=20] (a.corner 6);
  \path (a.corner 2) edge [bend right=20] (a.corner 6);

  \path (a.corner 2) edge [bend left=20] (a.corner 7);
  \path (a.corner 2) edge [bend right=20] (a.corner 7);

  \path (a.corner 3) edge [bend left=20] (a.corner 7);
  \path (a.corner 3) edge [bend right=20] (a.corner 7);

  \path (a.corner 3) edge [bend left=20] (a.corner 8);
  \path (a.corner 3) edge [bend right=20] (a.corner 8);

  \path (a.corner 4) edge [bend left=20] (a.corner 8);
  \path (a.corner 4) edge [bend right=20] (a.corner 8);

  \path (a.corner 4) edge [bend left=20] (a.corner 9);
  \path (a.corner 4) edge [bend right=20] (a.corner 9);

  \path (a.corner 5) edge [bend left=20] (a.corner 1);
  \path (a.corner 5) edge [bend right=20] (a.corner 1);

  \path (a.corner 5) edge [bend left=20] (a.corner 9);
  \path (a.corner 5) edge [bend right=20] (a.corner 9);

\end{tikzpicture}
\longleftarrow
\begin{tikzpicture}[baseline={([yshift=-0.6ex] current bounding box.center)}]
\node[draw=none,minimum size=2cm,regular polygon,regular polygon sides=27] (a) {};

\foreach \x in {1,2,...,27}
  \fill (a.corner \x) circle[radius=0.7pt];
  
\foreach \y\z in {1/4,2/5,3/6,4/7,5/8,6/9,7/10,8/11,9/12,10/13,11/14,12/15,13/16,14/17,15/18,16/19,17/20,18/21,19/22,20/23,21/24,22/25,23/26,24/27,25/1,26/2,27/3}
  \path (a.corner \y) edge (a.corner \z);
  
\foreach \y\z in {1/12,2/13,3/14,4/15,5/16,6/17,7/18,8/19,9/20,10/21,11/22,12/23,13/24,14/25,15/26,16/27,17/1,18/2,19/3,20/4,21/5,22/6,23/7,24/8,25/9,26/10,27/11}
  \path (a.corner \y) edge (a.corner \z); 
  
\foreach \y\z in {1/15,2/16,3/17,4/18,5/19,6/20,7/21,8/22,9/23,10/24,11/25,12/26,13/27,14/1,15/2,16/3,17/4,18/5,19/6,20/7,21/8,22/9,23/10,24/11,25/12,26/13,27/14}
  \path (a.corner \y) edge (a.corner \z);

\end{tikzpicture}
\, \, \longleftarrow \ldots
\end{equation*}
The power series $Q$ starts as follows 
$$Q(T) = (0.0111\ldots) T + (1.1020\ldots) T^{2} + (1.0200\ldots) T^{3} + \ldots \in \mathbb{Z}_{3}\llbracket T \rrbracket,$$
so we should have $\mu = 0$ and $\lambda = 3$.  We calculate
$$\kappa_{0} = 1, \kappa_{1} = 3^{3}, \kappa_{2} = 3^{6} \cdot 19^{2}  , \kappa_{3} = 3^{9} \cdot 19^{2} \cdot 703459^{2} , \ldots $$
and we have
$${\rm ord}_{3}(\kappa_{n}) =  3n,$$
for all $n \ge 0$.

\item Let $\ell = 13$ and $a_{1} = \sqrt{3} = 4.868\ldots, a_{2} = \sqrt{10} = 6.264\ldots \in \mathbb{Z}_{13}$.  Then, we get:
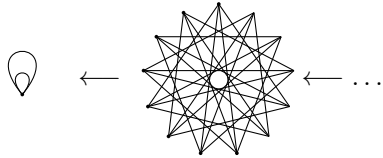
\begin{equation*}
\begin{tikzpicture}[baseline={([yshift=-1.7ex] current bounding box.center)}]
\node[draw=none,minimum size=2cm,regular polygon,regular polygon sides=1] (a) {};
\foreach \x in {1}
  \fill (a.corner \x) circle[radius=0.7pt];
\draw (a.corner 1) to [in=50,out=130,loop] (a.corner 1);
\draw (a.corner 1) to [in=50,out=130,distance = 0.5cm,loop] (a.corner 1);
\end{tikzpicture}
\longleftarrow \, \, \,
\begin{tikzpicture}[baseline={([yshift=-0.6ex] current bounding box.center)}]
\node[draw=none,minimum size=2cm,regular polygon,regular polygon sides=13] (a) {};

\foreach \x in {1,2,...,8}
  \fill (a.corner \x) circle[radius=0.7pt];
  
\foreach \y\z in {1/5,2/6,3/7,4/8,5/9,6/10,7/11,8/12,9/13,10/1,11/2,12/3,13/4}
  \path (a.corner \y) edge (a.corner \z);
  
\foreach \y\z in {1/7,2/8,3/9,4/10,5/11,6/12,7/13,8/1,9/2,10/3,11/4,12/5,13/6}
  \path (a.corner \y) edge (a.corner \z);
\end{tikzpicture}
\longleftarrow 
\ldots
\end{equation*}
The power series $Q$ starts as follows 
$$Q(T) = 13 T - 8 T^{2} + (3.4961\ldots) T^{3} + \ldots \in \mathbb{Z}_{13}\llbracket T \rrbracket,$$
so we should have $\mu = 0$ and $\lambda = 3$.  We calculate
$${\rm ord}_{13}(\kappa_{0}) = 0, {\rm ord}_{13}(\kappa_{1}) = 3, {\rm ord}_{13}(\kappa_{2}) = 6, {\rm ord}_{13}(\kappa_{3}) = 9, \ldots $$
We have
$${\rm ord}_{13}(\kappa_{n}) =  3n,$$
for all $n \ge 0$.

\end{enumerate}

\bibliographystyle{plain}
\bibliography{main}

\end{document}